\documentclass[12pt]{amsart}
\usepackage{amssymb}
\usepackage{amsfonts}
\usepackage{amssymb,latexsym}
\usepackage{enumerate}
\usepackage{mathrsfs}
\makeatletter
\@namedef{subjclassname@2010}{%
  \textup{2010} Mathematics Subject Classification}
\makeatother

  [2002/01/19 v2.2g %
    AMS font definitions%
  ]
\DeclareFontFamily{U}{euf}{}
\DeclareFontShape{U}{euf}{m}{n}{%
  <5><6><7><8><9>gen*eufm%
  <10><10.95><12><14.4><17.28><20.74><24.88>eufm10%
  }{}
\DeclareFontShape{U}{euf}{b}{n}{%
  <5><6><7><8><9>gen*eufb%
  <10><10.95><12><14.4><17.28><20.74><24.88>eufb10%
  }{}

  [2002/01/19 v2.2g %
    AMS font definitions%
  ]
\DeclareFontFamily{U}{msb}{}
\DeclareFontShape{U}{msb}{m}{n}{%
  <5><6><7><8><9>gen*msbm%
  <10><10.95><12><14.4><17.28><20.74><24.88>msbm10%
  }{}

  [2002/01/19 v2.2g %
    AMS font definitions%
  ]
\DeclareFontFamily{U}{msa}{}
\DeclareFontShape{U}{msa}{m}{n}{%
  <5><6><7><8><9>gen*msam%
  <10><10.95><12><14.4><17.28><20.74><24.88>msam10%
  }{}

\newtheorem{theorem}{Theorem}[section]
\newtheorem{lemma}[theorem]{Lemma}

\theoremstyle{definition}
\newtheorem{definition}[theorem]{Definition}

\newtheorem{remark}[theorem]{Remark}

\numberwithin{equation}{section}
\newcommand{\bm}[1]{\mbox{\boldmath{$#1$}}}

\begin{document}

\title[On a uniformly distributed
phenomenon] {On a uniformly distributed phenomenon in matrix groups}

\author{Su Hu and Yan Li}

\address{Department of Mathematics and Statistics, McGill University, 805 Sherbrooke St. West, Montr\'eal, Qu\'ebec H3A 2K6, Canada}
\email{hus04@mails.tsinghua.edu.cn, hu@math.mcgill.ca}

\address{Department of Applied Mathematics, China Agriculture University, Beijing 100083, China}
\email{liyan\_00@mails.tsinghua.edu.cn}


\begin{abstract}We show that  a classical uniformly distributed
phenomenon for an element and its inverse in
($\mathbb{Z}/n\mathbb{Z})^{*}$ also exists in
$\textrm{GL}_{n}(\mathbb{F}_{p})$. A
$\textrm{GL}_{n}(\mathbb{F}_{p})$ analogy of the uniform
distribution on modular hyperbolas has also been considered.
\end{abstract}

\maketitle

\def\C{\mathbb C_p}
\def\BZ{\mathbb Z}
\def\Z{\mathbb Z_p}
\def\Q{\mathbb Q_p}
\def\C{\mathbb C_p}
\def\BZ{\mathbb Z}
\def\Z{\mathbb Z_p}
\def\Q{\mathbb Q_p}
\def\psum{\sideset{}{^{(p)}}\sum}
\def\pprod{\sideset{}{^{(p)}}\prod}

\section{Introduction}
The distance between an element $x\in (\mathbb{Z}/n\mathbb{Z})^{*}$
and its inverse $x^{-1}~(\textrm{mod~}n)$ has been studied by many
authors~\cite{BK,Sh1,Sh2,Sh3,TZ,Z1,Z2}. Shparlinski~\cite{Sh4} gave a
survey of a variety of recent results about the distribution and
some geometric properties of points $(x,y)$ on modular hyperbolas
$xy \equiv a \pmod n$.

Denote by $\{x\}$ the fractional part of a real number $x$. Let

$$\begin{aligned}f_{n}:\ \left(\mathbb{Z}/n\mathbb{Z}\right)^*
&\rightarrow [0,1]\times[0,1]\\
 x&\mapsto\left(\left\{\frac{x}{n}\right\},\left\{\frac{x^{-1}}{n}\right\}\right).\end{aligned}
$$

By using the Erd\"os-Tur\'an-Koksma  inequality and the
Weil-Estermann inequality for Kloosterman sum, Beck and
Khan~\cite{BK} gave an elegant proof for the following classical
result.

\begin{theorem}\label{change}
Let $R\subset [0,1]^2$ be a measurable set having the following
property that for every $\epsilon > 0$, there exist two finite
collections of non-overlapping rectangles $R_{1},\ldots,R_{k}$ and
$R^{1},\ldots,R^{l}$ such that $\cup_{i=1}^{k}R_{i}\subseteq R
\subseteq \cup_{j=1}^{l}R^{j}$, ${\rm
area}\left(R/\cup_{i=1}^{k}R_{i}\right) < \epsilon$ and ${\rm
area}\left(\cup_{j=1}^{l}R^{j}/R\right) < \epsilon$. Then
$$\lim_{n\rightarrow\infty}\frac{{\rm cardinality}({\rm Image}(f_{n})\cap
R)}{\varphi(n)}={\rm area}\left(R\right).$$\end{theorem}

\begin{remark} Notice that our statement of the above theorem is slightly different from the statement in Beck and
Khan~\cite{BK}. The statement in~\cite{BK} is as follows:

``Let $R\subseteq [0,1]^{2}$ be a measurable set having the
following property that for every $\epsilon > 0$, there exists a
finite collection of non-overlapping rectangles
$\{R_{1},R_{2},\ldots, R_{k}\}$ such that
$\cup_{i=1}^{k}R_{i}\subseteq R ~~{\rm and}~~ {\rm
area}\left(R/\cup_{i=1}^{k}R_{i}\right) < \epsilon.$ Then
$$\lim_{n\rightarrow\infty}\frac{{\rm cardinality}({\rm Image}(f_{n})\cap
R)}{\varphi(n)}={\rm area}\left(R\right)" $$(see Theorem 2 of
~\cite{BK}).

The original assumption should be strengthened. Otherwise there is a
counterexample as follows:

Let $R_{1}= [0,1/2)^{2}$  and $R_{2}=\{(x,y)\in [0,1]^{2}\mid
x,y\in\mathbb{Q}\}.$
 Denote by
$R=R_{1}\cup R_{2}$. Since ${\rm area}\left(R_{2}\right)=0$, we have
${\rm area}\left(R\right)={\rm area}\left(R_{1}\right)=1/4$. So $R$
satisfies the conditions in the  statement of Theorem 2
in~\cite{BK}. Since the image of $f_{n}$ are rational points in
$[0,1]^{2}$ and $R$ contains all the rational points in $[0,1]^{2}$,
we have ${\rm Image}(f_{n})\cap R =\varphi(n)$ for any positive
integer $n$, thus $$\lim_{n\rightarrow\infty}\frac{{\rm
cardinality}({\rm Image}(f_{n})\cap R)}{\varphi(n)}=1. $$ But ${\rm
area}\left(R\right)={\rm area}\left(R_{1}\right)=1/4$, so
$$\lim_{n\rightarrow\infty}\frac{{\rm cardinality}({\rm
Image}(f_{n})\cap R)}{\varphi(n)} \neq {\rm area}\left(R\right) .$$

Notice that, the new conditions in Theorem \ref{change} are quite
natural. Numerous types of regions satisfy the conditions of Theorem
\ref{change} such as polygons, disks, annuli lying in the unit
square.
\end{remark}

Beck and Khan~\cite[p.\,150]{BK}  remarked that: ``In all likelihood
this theorem dates back to the late 20's and early 30's and was
known to mathematicians such as Davenport, Estermann, Kloosterman,
Salie."

Let
$\mathbb{F}_{p}=\mathbb{Z}/p\mathbb{Z}=\{\bar{0},\bar{1},\ldots,\overline{p-1}\}$
be the finite field with $p$ elements, $M_{n}(\mathbb{F}_{p})$ be
the set of all $n \times n $ matrices over $\mathbb{F}_{p}$,
${\rm GL}_{n}(\mathbb{F}_{p}), \textrm{SL}_{n}(\mathbb{F}_{p})$
and  $\mathcal{Z}_{n}(\mathbb{F}_{p})$ be the group of invertible
matrices, the group of matrices of determinant 1 and the set of
singular matrices, respectively, where all matrices are from
$M_{n}(\mathbb{F}_{p})$.

In this paper, by using bounds of Ferguson, Hoffman, Ostafe, Luca
and Shparlinski~\cite{Sh} for the matrix analogue of classical
Kloosterman sums (see Lemma~\ref{l1} below), we show that the
above mentioned uniformly distributed phenomenon also exists in
${\rm GL}_{n}(\mathbb{F}_{p})$.

For $A=\left(\overline{a_{ij}}\right)\in
{\rm GL}_{n}(\mathbb{F}_{p}),A^{-1}=\left(\overline{b_{ij}}\right)$
  denotes the inverse of $A$.

Let \begin{equation}~\label{def1}g_{p}:\
{\rm GL}_{n}(\mathbb{F}_{p}) \rightarrow \underbrace{[0,1] \times
[0,1] \cdots \times [0,1]}_{2n^{2}}
\end{equation}
$$
A=\left(\overline{a_{ij}}\right)\mapsto\left(
  \begin{array}{cccccc}
    \frac{\displaystyle a_{11}}{\displaystyle p}, & \ldots, & \frac{\displaystyle a_{1n}}{\displaystyle p}, &\frac{\displaystyle b_{11}}{\displaystyle p}, & \ldots, & \frac{\displaystyle b_{1n}}{\displaystyle p}  \\
    \frac{\displaystyle a_{21}}{\displaystyle p}, & \ldots, & \frac{\displaystyle a_{2n}}{\displaystyle p}, &\frac{\displaystyle b_{21}}{\displaystyle p}, & \ldots, & \frac{\displaystyle b_{2n}}{\displaystyle p}    \\
    \vdots & \vdots & \vdots & \vdots & \vdots & \vdots \\
   \frac{\displaystyle a_{n1}}{\displaystyle p}, & \ldots, & \frac{\displaystyle a_{nn}}{\displaystyle p}, &\frac{\displaystyle b_{n1}}{\displaystyle p}, & \ldots, & \frac{\displaystyle b_{nn}}{\displaystyle p}   \\
  \end{array}
\right)$$

In fact, we prove the following theorem.

\begin{theorem}~\label{ma}
Let $R\subset [0,1]^{2n^{2}}$ be a measurable set having the
following property that for every $\epsilon > 0$, there exist two
finite collections of non-overlapping rectangles
$R_{1},\ldots,R_{k}$ and $R^{1},\ldots,R^{l}$ such that
$\cup_{i=1}^{k}R_{i}\subset R \subset \cup_{j=1}^{l}R^{j}$, ${\rm
area}\left(R/\cup_{i=1}^{k}R_{i}\right) < \epsilon$ and ${\rm
area}\left(\cup_{j=1}^{l}R^{j}/R\right) < \epsilon$. Then

$$\lim_{p\rightarrow\infty}\frac{{\rm cardinality}({\rm Image}(g_{p})\cap
R)}{\#{\rm GL}_{n}(\mathbb{F}_{p})}={\rm area}\left(R\right).$$
\end{theorem}
\begin{remark}  A
$\textrm{GL}_{n}(\mathbb{F}_{p})$ analogy of the uniform
distribution on modular hyperbolas can also be established using the
same procedure for the proof of the above theorem (see
Remark~\ref{r1} below).
\end{remark}

Furthermore, let \begin{equation} ~\label{def2}h_{p}:\ {\rm
GL}_{n}(\mathbb{F}_{p}) \rightarrow \underbrace{[0,1] \times [0,1]
\cdots \times [0,1]}_{n^{2}}
\end{equation}
$$
A=\left(\overline{a_{ij}}\right)\mapsto\left(
  \begin{array}{ccc}
    \frac{\displaystyle a_{11}}{\displaystyle p}, & \ldots, & \frac{\displaystyle a_{1n}}{\displaystyle p}  \\
    \frac{\displaystyle a_{21}}{\displaystyle p}, & \ldots, & \frac{\displaystyle a_{2n}}{\displaystyle p}     \\
    \vdots & \vdots & \vdots  \\
   \frac{\displaystyle a_{n1}}{\displaystyle p}, & \ldots, & \frac{\displaystyle a_{nn}}{\displaystyle p}  \\
  \end{array}
\right)$$
\begin{equation} ~\label{def3}s_{p}:\
\textrm{SL}_{n}(\mathbb{F}_{p}) \rightarrow \underbrace{[0,1] \times
[0,1] \cdots \times [0,1]}_{n^{2}}
\end{equation}
$$
A=\left(\overline{a_{ij}}\right)\mapsto\left(
  \begin{array}{ccc}
    \frac{\displaystyle a_{11}}{\displaystyle p}, & \ldots, & \frac{\displaystyle a_{1n}}{\displaystyle p}  \\
    \frac{\displaystyle a_{21}}{\displaystyle p}, & \ldots, & \frac{\displaystyle a_{2n}}{\displaystyle p}     \\
    \vdots & \vdots & \vdots  \\
   \frac{\displaystyle a_{n1}}{\displaystyle p}, & \ldots, & \frac{\displaystyle a_{nn}}{\displaystyle p}  \\
  \end{array}
\right)$$

Using the same methods, and the bounds of Ferguson, Hoffman,
Ostafe, Luca and Shparlinski~\cite{Sh} on character sums along
$\mathcal {Z}_{n}(\mathbb{F}_{p})$ and ${\rm
SL}_{n}(\mathbb{F}_{p})$ (see Lemmas~\ref{l2},~\ref{l3}~
and~\ref{l4} below), we can also obtain the following two results.

\begin{theorem}\label{s2}Let $R\subseteq [0,1]^{n^{2}}$ be a measurable set having the
same property as in Theorem \ref{ma}. Then

$$\lim_{p\rightarrow\infty}\frac{{\rm cardinality}({\rm Image}(h_{p})\cap
R)}{\#{\rm GL}_{n}(\mathbb{F}_{p})}={\rm area}\left(R\right).$$
\end{theorem}

\begin{theorem}\label{s3}Assumption as above, then

$$\lim_{p\rightarrow\infty}\frac{{\rm cardinality}({\rm Image}(s_{p})\cap
R)}{\#{\rm SL}_{n}(\mathbb{F}_{p})}={\rm area}\left(R\right).$$
\end{theorem}

\section{Preliminaries}\label{sec:sec2}
We need some lemmas to prove the main theorems.

First, we recall some results in~\cite{Sh}.

Given two matrices $U=(u_{ij}),X=(x_{ij})\in M_{n}(\mathbb{F}_{p})$,
their product is defined by $$U\cdot X
= \sum_{\substack{1\leq i\leq n \\
1\leq j\leq n}}u_{ij}x_{ij}$$ (see \cite[p.\,503]{Sh}).

Let $q$ be a power of a prime number, $\Psi$ be a fixed nonprincipal
additive character of $\mathbb{F}_{q}$. For $\mathcal{M},U,V\in
M_{n}(\mathbb{F}_{q})$, let
$$K({\rm GL}_{n}(\mathbb{F}_{q}),U,V,\mathcal{M}) =
\sum_{X\in{\rm GL}_{n}(\mathbb{F}_{q})}\Psi(U \cdot X + V
\cdot(\mathcal{M}X^{-1}))$$ be the matrix analogue of classical
Kloosterman sums (see \cite[p.\,505]{Sh}) and
\begin{equation*}\begin{aligned}\\S({\rm GL}_{n}(\mathbb{F}_{q}),U) &=
\sum_{X\in{\rm GL}_{n}(\mathbb{F}_{q})}\Psi(U \cdot X
),\\S(\textrm{SL}_{n}(\mathbb{F}_{q}),U) &=
\sum_{X\in\textrm{SL}_{n}(\mathbb{F}_{q})}\Psi(U \cdot X ),
\\S(\mathcal{Z}_{n}(\mathbb{F}_{q}),U)& =
\sum_{X\in\mathcal{Z}_{n}(\mathbb{F}_{q})}\Psi(U \cdot
X).\end{aligned}\end{equation*}
 The authors in~\cite{Sh} obtained the following
results.

\begin{lemma}~\label{l1} (see \cite[Lemma 5]{Sh}) Uniformly over all matrices
$U,V\in M_{n}(\mathbb{F}_{q})$ among which at least one is a nonzero
matrix, and $\mathcal{M}\in{\rm GL}_{n}(\mathbb{F}_{q})$, we have
$$K({\rm GL}_{n}(\mathbb{F}_{q}),U,V,\mathcal{M})\ll q^{n^{2}-1/2},$$
where the implied constant in the symbol $``\ll"$ depends only on
$n$.\end{lemma}

\begin{lemma}~\label{l2} (see \cite[Lemma 3]{Sh}) Uniformly over all nonzero matrices
$U\in M_{n}(\mathbb{F}_{q})$, we have
$$S(\mathcal{Z}_{n}(\mathbb{F}_{q}),U)= O(q^{n^{2}-5/2}),$$
where the implied constant in the symbol $``O"$ depends only on
$n$.\end{lemma}

\begin{lemma}~\label{l3}  Uniformly over all nonzero matrices
$U\in\textrm{M}_{n}(\mathbb{F}_{q})$, we have
$$S({\rm GL}_{n}(\mathbb{F}_{q}),U)= O(q^{n^{2}-5/2}),$$
where the implied constant in the symbol $``O"$ depends only on
$n$.\end{lemma}
\begin{proof} For any nonzero matrix
$U\in M_{n}(\mathbb{F}_{q})$, $\bar{\Psi}(X)=\Psi(U \cdot
X)$ is also a nontrivial additive character on
$M_{n}(\mathbb{F}_{q})$, so we have
$$S(\mathcal{Z}_{n}(\mathbb{F}_{q}),U)+S({\rm GL}_{n}(\mathbb{F}_{q}),U)=\sum_{X\in M_{n}(\mathbb{F}_{q})}\Psi(U \cdot X
)=0.$$ From Lemma~\ref{l2}, we obtain the result.
\end{proof}

\begin{lemma}~\label{l4} (see \cite[Lemma 4]{Sh}) Uniformly over all nonzero matrices
$U\in\textrm{M}_{n}(\mathbb{F}_{q})$, we have
$$S({\rm SL}_{n}(\mathbb{F}_{q}),U)= O(q^{n^{2}-2}),$$
where the implied constant in the symbol $``O"$ depends only on
$n$.\end{lemma}

\begin{remark}Recently, we obtained explicit expressions of $S({\rm
GL}_{n}(\mathbb{F}_{q}),U)$ and $S({\rm
SL}_{n}(\mathbb{F}_{q}),U)$.(See \cite{LH} Theorems 2.1 and
2.4).  Such expressions only involve Gauss sums and Kloosterman
sums. As a consequence, we got
\begin{equation*}
\begin{aligned}S({\rm GL}_{n}(\mathbb{F}_{q}),U)&=
O(q^{n^{2}-n}),\\ S({\rm SL}_{n}(\mathbb{F}_{q}),U)=
O(q^{n^{2}-n})&=O(\max\{q^{n^{2}-n-1},q^{(n^{2}-1)/2}\}).\end{aligned}
\end{equation*} (See \cite{LH} Corollaries 2.2 and
2.5).
Our bounds
 improved the bounds in Lemmas~\ref{l3} and \ref{l4}. (See \cite{LH} Remarks 2.3 and 2.6).
\end{remark}

Next we recall Erd\"os-Tur\'an-Koksma's inequality for the
discrepancy of sequences.

Let $B=[a_{1},b_{1})\times\cdots\times [a_{k},b_{k})\subseteq
[0,1)^{k}$ be a rectangle, (${\bm x_{n}}$) be a sequence in
$[0,1)^{k}$, and $A(B,N,{\bm x_{n}})$ be the number of points ${\bm
x_{n}}$, $1\leq n\leq N$, such that $ {\bm x_{n}}\in B$, i.e.
$$A(B,N,{\bm x_{n}})=\sum_{n=1}^{N}\chi_{B}({\bm x_{n}}),$$
where $\chi_{B}$ is the characteristic function of $B$.

\begin{definition} (see \cite[p.\,4]{Ti}) Let $\{{\bm x_{1}},\ldots,{\bm x_{N}}\}$ be a
finite sequence of points in $[0,1)^{k}$. Then the number
$$D_{N}=D_{N}({\bm x_{1}},\ldots,{\bm x_{N}})=
\textrm{sup}_{B\subseteq [0,1)^{k}}\Big|\frac{A(B,N,{\bm x_{n}})}{N} -
\textrm{area}(B)\Big|$$ is called the discrepancy of the given sequence,
where $B$ runs over all rectangles located in $[0,1)^{k}$.
\end{definition}

Set $e(x)=\textrm{exp}(2\pi i x)$ and denote the usual inner product
in $\mathbb{R}^{k}$ by ${\bm
x}\cdot\bm{y}=\sum_{i=1}^{k}x_{i}y_{i}$.

The Erd\"os-Tur\'an-Koksma inequality provides an upper bound for
the discrepancy.

\begin{lemma}~\label{l5} (see \cite[p.\,15]{Ti} or \cite[p.\,63]{Nie}) Let $\{{\bm x_{1}},\ldots,{\bm x_{N}}\}$ be a
finite sequence of points in $[0,1)^{k}$ and $H$ an arbitrary
positive integer. Then

$$D_{N}\leq
\Big(\frac{3}{2}\Big)^{k}\Big(\frac{2}{H+1}+\sum_{0 < ||{\bm
h}||_{\infty}\leq H} \frac{1}{r({\bm
h})}\Big|\frac{1}{N}\sum_{n=1}^{N} e({\bm h}\cdot{\bm
x_{n}})\Big|\Big),$$ where $r({\bm h})=\prod_{i=1}^{k}{\rm
max}\{1,|h_{i}|\}$ and $||{\bm h}||_{\infty}={\rm max}\{\ |h_{i}|\
|1\leq i\leq k\}$, for ${\bm h}=(h_{1},\ldots,h_{k}) \in
\mathbb{Z}^{k}.$\end{lemma}

\section{Proofs of  main results}
\textbf{Proof of Theorem~\ref{ma}:}

We only need to prove the case that
$R=[a_{1},b_{1})\times\cdots\times [a_{k},b_{k})\subset [0,1)^{k}$
is a rectangle. The reason is as follows:

Assume the theorem holds for $R$ being a rectangle. Let $R$ be a
measurable set as in the assumptions. For every $\epsilon > 0$, let
$R_{1},\ldots,R_{k}$ and $R^{1},\ldots,R^{l}$ be two finite
collections of non-overlapping rectangles such that
$$\cup_{i=1}^{k}R_{i}\subseteq R \subseteq \cup_{j=1}^{l}R^{j},\ {\rm
area}\left(R/\cup_{i=1}^{k}R_{i}\right) < \epsilon\ {\rm and}\ {\rm
area}\left(\cup_{j=1}^{l}R^{j}/R\right) < \epsilon.$$ Then
\begin{equation}~\label{p1}\begin{aligned}\sum_{i=1}^{k}{\rm area}\left(R_i\right)&\leq{\rm
area}\left(R\right)\leq\sum_{j=1}^{l}{\rm area}\left(R^j\right),
\\\sum_{i=1}^{k}\frac{\#(\textrm{Image}(g_{p})\cap R_i)}{\#{\rm
GL}_{n}(\mathbb{F}_{p})}&\leq\frac{\#(\textrm{Image}(g_{p})\cap
R)}{\#{\rm
GL}_{n}(\mathbb{F}_{p})}\leq\sum_{j=1}^{l}\frac{\#(\textrm{Image}(g_{p})\cap
R^j)}{\#{\rm GL}_{n}(\mathbb{F}_{p})}.
\end{aligned}\end{equation}
Taking $p$ sufficiently large, the left hand (resp. right hand)
sides of the above two inequalities are sufficiently close, i.e,
there exists $M$ such that if $p>M$ then
\begin{equation}~\label{p2}\begin{aligned}\Big|\sum_{i=1}^{k}\frac{\#(\textrm{Image}(g_{p})\cap R_i)}{\#{\rm
GL}_{n}(\mathbb{F}_{p})}-\sum_{i=1}^{k}{\rm
area}\left(R_i\right)\Big|&<\epsilon,
\\\Big|\sum_{j=1}^{l}\frac{\#(\textrm{Image}(g_{p})\cap R^j)}{\#{\rm
GL}_{n}(\mathbb{F}_{p})}-\sum_{j=1}^{l}{\rm
area}\left(R^j\right)\Big|&<\epsilon.
\end{aligned}\end{equation}
Since \begin{equation}~\label{p3}0\leq \sum_{j=1}^{l}{\rm
area}\left(R^j\right)-\sum_{i=1}^{k}{\rm
area}\left(R_i\right)<2\epsilon,\end{equation} then from
inequalities~(\ref{p1}), (\ref{p2}), and~(\ref{p3}), for $p>M$, we
have $$\Big|\frac{\#(\textrm{Image}(g_{p})\cap R)}{\#{\rm
GL}_{n}(\mathbb{F}_{p})}-\textrm{area}(R)\Big|<4\epsilon.
$$
This implies that
$$\lim_{p\rightarrow\infty}\frac{{\rm cardinality}({\rm Image}(g_{p})\cap
R)}{\#{\rm GL}_{n}(\mathbb{F}_{p})}={\rm area}\left(R\right).$$

Now we prove the fundamental case in which $R$ is the rectangle\\
$[a_{1},b_{1})\times\cdots\times [a_{k},b_{k})$.

 In Lemma \ref{l5}, viewing the
points ${\bm x}$ in $\textrm{Image}(g_{p})$ as a sequence in
$[0,1)^{k}$, where $N=\#{\rm GL}_{n}(\mathbb{F}_{p})$ and $k =
2n^{2}$, we have

\begin{equation}~\label{t1}
\begin{aligned}&\Big|\frac{\textrm{cardinality}(\textrm{Image}(g_{p})\cap
R)}{\#{\rm GL}_{n}(\mathbb{F}_{p})}-\textrm{area}(R)\Big|\\&\ll
\frac{2}{H+1}+\sum_{0 < ||{\bm h}||_{\infty}\leq H}
\frac{1}{r({\bm h})}\Big|\frac{1}{\#{\rm
GL}_{n}(\mathbb{F}_{p})}\sum_{{\bm x}\in\textrm{Image}(g_{p})}
e({\bm h}\cdot{\bm x})\Big|,\end{aligned}
\end{equation}
where $r({\bm h})=\prod_{\substack{1\leq i\leq n \\
1\leq j\leq n}}\textrm{max}\{1,|h_{ij}|\}$ for
$${\bm h}=\left(
  \begin{array}{cccccc}
    h_{11}, & \ldots, & h_{1n}, &h_{(n+1)1} ,&\ldots, &h_{(n+1)n}  \\
    h_{21},&\ldots,&h_{2n}, & h_{(n+2)1} ,&\ldots, &h_{(n+2)n}  \\
    \vdots & \vdots & \vdots & \vdots & \vdots & \vdots \\
    h_{n1} ,&\ldots,
&h_{nn} , & h_{(2n)1} ,&\ldots,& h_{(2n)n} \\
  \end{array}
\right)$$
 in
$\mathbb{Z}^{2n^{2}}$.

Since $\#{\rm GL}_{n}(\mathbb{F}_{p})=p^{n^{2}}+O(p^{n^{2}-1})$, if
${\bm h}$ modulo $p$ is nonzero, then by Lemma~\ref{l1} (taking
$\mathcal{M}=I,U=(\overline{h_{ij}})_{\substack{1\leq i\leq n\\1\leq
j\leq n}}, V=(\overline{h_{(n+i)j}})_{\substack{1\leq i\leq n\\1\leq
j\leq n}}, X=A$), we will get

\begin{equation}~\label{ta} \frac{1}{r({\bm h})}\Big|\frac{1}{\#{\rm
GL}_{n}(\mathbb{F}_{p})}\sum_{{\bm x}\in\textrm{Image}(g_{p})}
e({\bm h}\cdot{\bm x})\Big|\ll p^{-1/2}\end{equation}

Notice that
\begin{equation}~\label{ta2}\#\{{\bm h}\in\mathbb{Z}^{2n^2}|0 < ||{\bm h}||_{\infty}\leq H\}\ll H^{2n^2}.\end{equation}

From (\ref{t1}), (\ref{ta}), (\ref{ta2}), and taking
$H=\lfloor p^{1/(2(2n^2+1))}\rfloor$, we get
\begin{equation}~\label{t4}
\begin{aligned}&\Big|\frac{\textrm{cardinality}(\textrm{Image}(g_{p})\cap
R)}{\#{\rm GL}_{n}(\mathbb{F}_{p})}-\textrm{area}(R)\Big|\\&\ll
\frac{2}{H+1}+H^{2n^2}p^{-1/2}\\&\ll
p^{-1/(2(2n^2+1))},\end{aligned}
\end{equation}

Letting $\lim_{p\to\infty}$ in (\ref{t4}), we get our result.

\begin{remark}~\label{r1}  For any fixed $C=\left(\overline{a_{ij}}\right)\in
{\rm GL}_{n}(\mathbb{F}_{p})$, we consider the matrix equation
$BA=C$, where $A=\left(\overline{a_{ij}}\right)$,
$B=\left(\overline{b_{ij}}\right)$ in ${\rm
GL}_{n}(\mathbb{F}_{p})$.

Let
\begin{equation}~\label{def1+}\tilde{g}_{p}:\ {\rm
GL}_{n}(\mathbb{F}_{p}) \rightarrow \underbrace{[0,1] \times [0,1]
\cdots \times [0,1]}_{2n^{2}}
\end{equation}
$$
A=\left(\overline{a_{ij}}\right)\mapsto\left(
  \begin{array}{cccccc}
    \frac{\displaystyle a_{11}}{\displaystyle p}, & \ldots, & \frac{\displaystyle a_{1n}}{\displaystyle p}, &\frac{\displaystyle b_{11}}{\displaystyle p}, & \ldots, & \frac{\displaystyle b_{1n}}{\displaystyle p}  \\
    \frac{\displaystyle a_{21}}{\displaystyle p}, & \ldots, & \frac{\displaystyle a_{2n}}{\displaystyle p}, &\frac{\displaystyle b_{21}}{\displaystyle p}, & \ldots, & \frac{\displaystyle b_{2n}}{\displaystyle p}    \\
    \vdots & \vdots & \vdots & \vdots & \vdots & \vdots \\
   \frac{\displaystyle a_{n1}}{\displaystyle p}, & \ldots, & \frac{\displaystyle a_{nn}}{\displaystyle p}, &\frac{\displaystyle b_{n1}}{\displaystyle p}, & \ldots, & \frac{\displaystyle b_{nn}}{\displaystyle p}   \\
  \end{array}
\right).$$ The same procedure can show Theorem~\ref{ma} is also
established for $\tilde{g}_{p}$. The only change is that we take
$\mathcal{M}=C,U=(h_{ij})_{\substack{1\leq i\leq n\\1\leq j\leq n}},
V=(h_{(n+i)j})_{\substack{1\leq i\leq n\\1\leq j\leq n}}$ and $X=A$
in Lemma~\ref{l1}.  This can be viewed as a
$\textrm{GL}_{n}(\mathbb{F}_{p})$ analogy of the uniform
distribution on modular hyperbolas.

\end{remark}

\begin{remark}  As pointed out by the referee, we may also consider classes of sets for which the desired epsilon-approximation can be achieved with $O(\epsilon^{-\alpha})$
 rectangles and obtain an explicit bound on the discrepancy, depending on $\alpha$. A result by Wolfgang
 M. Schmidt in \cite{Sch} may give an explicit $\alpha$ for say convex
sets.
\end{remark}

\textbf{Proof of Theorem~\ref{s2}:}

The method is exactly the same as in Theorem~\ref{ma}, except that
we replace the estimation of $K({\rm
GL}_{n}(\mathbb{F}_{p}),U,V,\mathcal{M})$ by the estimation of
$S({\rm GL}_{n}(\mathbb{F}_{p}),U)$ (see Lemma \ref{l3}).

\textbf{Proof of Theorem~\ref{s3}:}

By \cite[Theorem 3.15~(iiib)]{Rose}, we have

\begin{equation}\label{new}
\begin{aligned}\#\textrm{SL}_{n}(\mathbb{F}_{p})&=\frac{p^{n(n-1)/2}}{p-1}\cdot\prod_{j=1}^{n}(p^{j}-1)\\&=p^{n(n-1)/2}(p^{2}-1)\cdots(p^{n}-1)
\\&=p^{n^{2}-1}+O(p^{n^{2}-3})
\end{aligned}\end{equation}

The method to prove this theorem is exactly the same as in
Theorem~\ref{ma}, except that we replace the estimation of $K({\rm
GL}_{n}(\mathbb{F}_{p}),U,V,\mathcal{M})$ by the estimation of
$S(\textrm{SL}_{n}(\mathbb{F}_{p}),U)$ (see Lemma \ref{l4}) and
notice that formula (\ref{new}) holds.

\section{Further discussions}
In this section, we discuss the following questions:

Is the image of $\textrm{SL}_{n}(\mathbb{F}_{p})$ under $g_{p}$ (see
(\ref{def1})), uniformly distributed in $[0,1]^{2n^{2}}$?

The case $n=1$ is trivial.

For $n=2$, since
$$\left(
    \begin{array}{cc}
      \bar{a} & \bar{b} \\
      \bar{c} & \bar{d} \\
    \end{array}
  \right)^{-1}=\left(
                 \begin{array}{cc}
                   \bar{d} & -\bar{b} \\
                   -\bar{c} & \bar{a}\\
                 \end{array}
               \right),\ {\rm where}\ \left(
    \begin{array}{cc}
      \bar{a} & \bar{b} \\
      \bar{c} & \bar{d} \\
    \end{array}
  \right)\in \textrm{SL}_{2}(\mathbb{F}_{p}),
$$
one can easily find a nonzero vector
$$
{\bm h}=\left(
  \begin{array}{cccccc}
    h_{11}, & h_{12},&h_{13}, & h_{14}   \\
     h_{21}, & h_{22},&h_{23}, & h_{24}
  \end{array}
\right)\in\mathbb{Z}^{8}$$

such that
$${\bm x}\mapsto e({\bm h}\cdot{\bm x})=1,\ {\rm for\ all}\ {\bm x}\in g_{p}(\textrm{SL}_{2}(\mathbb{F}_{p})).$$
For example, take arbitrary nonzero ${\bm h}$ with
$h_{11}+h_{24}=0$, $h_{12}-h_{14}=0$, $h_{13}+h_{22}=0$ and
$h_{21}-h_{23}=0$.

Hence
$$\lim_{p\rightarrow\infty}\frac{1}{\#{\rm SL}_{2}(\mathbb{F}_{p})}\sum\limits_{{\bm x}\in g_{p}(\textrm{SL}_{2}(\mathbb{F}_{p})
)}e({\bm h}\cdot{\bm x})\neq\int_{[0,1]^{8}}e({\bm h}\cdot{\bm x})d{\bm x}.$$
This implies that
the image of $\textrm{SL}_{2}(\mathbb{F}_{p})$ under $g_{p}$ is not uniformly distributed, that is,
Theorem \ref{ma} does not hold for $\textrm{SL}_{2}(\mathbb{F}_{p})$.

For the case $n\geq 3$, we conjecture that elements and their
inverses in $\textrm{SL}_{n}(\mathbb{F}_{p})$ are also uniformly
distributed.

\begin{remark}As in the proof of \cite[Lemma 4]{Sh}, the above
conjecture may be treated within a general theory of bounds of
exponential sums along varieties. Thus, as pointed out by the
referee, the estimations in \cite{F1,FK,K1,SSk,S1} may be relevant
to this consideration. In this remark, we shall mention and discuss
some of the results in these works, and we hope the above conjecture
may be solved in the future along this way.

Let $V\subset\mathbf{A}_{\mathbf{Z}}^{n}$ be a closed subscheme of
dimension $\leq d$  defined by the vanishing of several
polynomials, let $f(X_{1},\cdots,X_{n})$ be a polynomial defined
as a function on $V$, let $\psi$ be a nontrivial additive
character of $\mathbb{F}_{p}$ , and let
$\bm{h}=(h_{1},\cdots,h_{n})\in\mathbf{A}^{n}(\mathbb{F}_{p})$.
We set $$S_{V,f}(\bm{h};p):=\sum_{(x_{1},\cdots,x_{n})\in
V(\mathbb{F}_{p})}\psi(f(x_{1},\cdots,x_{n})+h_{1}x_{1}+\cdots+h_{n}x_{n}).$$
The ``good" bound which is expected is
$$S_{V,f}(\bm{h};p)\leq C(V,f)p^{d/2},$$
which is essentially equivalent to the Riemann Hypothesis for an
appropriate $L$-function over the finite field $\mathbb{F}_{p}$.

In~\cite{KG}, using the formalism of perversity and the properties
of the geometric Fourier transform, Katz and Laumon proved that,
for $f\equiv 0$ and under an extremely mild hypothesis on $V$  (in
particular, no smoothness assumptions are required), and for $p$
large enough (depending on $V$), the ``good" bound holds for all
$\bm{h}$  lying outside the set of $\mathbb{F}_{p}$-points of
a codimension-one subscheme
$X_{1}\subset\mathbf{A}_{\mathbf{Z}}^{n}$.

Let $f=0$. In~\cite{F1}, Fouvry investigated the structure of this
exceptional locus $X_{1}$ on which the good bound does not hold,
and in particular at the ``size" of the intermediate subsets
$(X_{j}$, say) of the $\bm{h}$ for which the bound for
$S_{V,f}(\bm{h};p)$ deviates from the good one by a factor of
at least $p^{j/2}$. Furthermore, he also applied these results to
prove several new results on the distribution of rational points
of varieties over finite fields.

In the case of  of a general $f$, Fouvry and Katz~\cite{FK}
established the existence of a decreasing filtration
$\cdots\subset X_{j}\subset X_{j-1}\subset\cdots\subset
X_{1}\subset X_{0}:=\mathbf{A}_{\mathbf{Z}}^{n}$, by closed
subschemes of codimension $j$, such that for $p$ large enough and
$\bm{h}\not\in X_{j}(\mathbb{F}_{p})$, the following bound
holds: $$S_{V,f}(\bm{h};p)\leq C(V,f)p^{d/2+(j-1)/2}.$$ Note
that in the above statement, there is essentially no assumption on
the regularity of $V$  or $f$. When $f\equiv 0$, $V_\textbf{C}$ is
smooth, and an extra geometrical condition is satisfied (i.e. the
non-vanishing of a certain ``A -number" which is the rank of an
$l$-adic sheaf attached to the situation), they also improved the
above bound as  $$S_{V,f}(\bm{h};p)\leq
C(V,f)p^{\textrm{sup}(d/2, d/2+(j-2)/2)},$$ for $\bm{h}
\not\in X_{j}(\mathbb{F}_{p}) $. Furthermore, they also presented
several Diophantine applications, in particular, they improved
some of previous results by Fouvry in~\cite{F2}.

Let $k$ be a finite field, let $X/k$ be a closed subscheme of the
projective space $\bold P^{N}$ of dimension $n\geq 1$, defined by
a set of polynomials of fixed degree $D_{1},\cdots, D_{r}$, let
$L\in H^0(X,\mathcal{O}(1))$ be a linear form on $X$ and for a
fixed integer $d$, let $H\in H^0(X,\mathcal{O}(d))$ be a form of
degree $d$; set $X\cap L=X\cap\{L=0\}$ and $V=X-X\cap L$; and let
$\psi\colon k\rightarrow\mathbb{C}^\times$ be a nontrivial
additive character. Denote by $$S:= \sum_{x\in
V(k)}\psi((H/L^d)(x)).$$ In \cite{K1}, Katz proved the sharp upper
bound $|S|\leq C|k|^{n/2}$ (for some constant
$C=C(N,d,D_{1},\cdots, D_{r})$ depending on $N,d,D_{1},\cdots,
D_{r}$ only), under the assumptions that ${\rm char}\,k$ is
coprime with $d$, that $X$, $X\cap L$ and $X\cap L\cap H$ are all
nonsingular, and that $X\cap L$ and $X\cap L\cap H$ have
codimension one and two, respectively, in $X$. In~\cite{K2}, he
further removed many of the smoothness assumptions made above;
more precisely, assuming either ${\rm
H}_{1}\colon$``$X\otimes_{k}{\overline k}$ is irreducible and
integral'' or ${\rm H}'_{1}\colon $ ``$X$ is Cohen-Macaulay and
equidimensional'' and also assuming ${\rm H}_{2}\colon $ ``the
scheme $X\cap L\cap H$ has codimension 2'', he obtained the upper
bound $\vert S\vert\leq C\vert k\vert^{(n+\delta+1)/2}$, where
$\delta$ is the dimension of the singular locus of $X\cap L\cap
H$, at least when ${\rm char}\,k$ is large enough. He also proved
a slightly weaker result, valid in any characteristic coprime with
$d$ under the same assumptions (${\rm H}_{1}$ or ${\rm H}'_{1}$,
and ${\rm H}_{2}$). Furthermore, he also applied these general
bounds to the case of hypersurfaces.

Shparlinski and  Skorobogatov~\cite{SSk} estimated the modulus of
exponential sums over the variety of dimension $n-s$  defined by a
system of s forms in $n$ variables, with a linear form in the
exponent. They also  applied the estimations to the study of the
distribution of rational points on such a variety defined over a
finite field or the field of rationals.

Using the results of Deligne's Weil Conjectures II and a
generalization of Lefschetz hyperplane theorem to singular
varieties, Skorobogatov~\cite{S1} estimated exponential sums with
additive character along an affine variety given by a system of
homogeneous equations, with a homogeneous function in the
exponent. He applied his estimate to obtain an upper bound for the
number of integer solutions of a system of homogeneous equations
in a box. Furthermore, he also applied his estimate
 to uniform distribution of
solutions of a system of homogeneous congruences modulo a prime in
the following sense: the portion of solutions in a box is
proportional to the volume of the box, provided the box is not
very small.

\end{remark}

\textbf{Acknowledgement} The authors are grateful to the anonymous
referee for his/her helpful comments.

Y. Li is supported by the National Natural Science Foundation of
China (Grant No. 11101424 and Grant No. 11071277).

\end{document}